\documentclass{amsart}

\usepackage{amssymb}
\usepackage[
    pagebackref,            
    colorlinks=true,        
    linkcolor=magenta,      
    citecolor=green,      
    urlcolor=blue,          
    bookmarks=true
]{hyperref}

\usepackage{cite}
\usepackage[all]{xy}
\usepackage{enumerate}
\usepackage{enumitem}
\usepackage{mathrsfs}
\usepackage{color}
\usepackage{tikz}
\usepackage{microtype} 
\usepackage{geometry}
\newtheorem{theorem}{Theorem}[section]
\newtheorem{lemma}[theorem]{Lemma}
\newtheorem{corollary}[theorem]{Corollary}

\theoremstyle{definition}

\newtheorem{remark}{Remark}[section]

\numberwithin{equation}{section}

\DeclareMathOperator{\Hess}{Hess}

\DeclareMathOperator{\Ric}{Ric}
\DeclareMathOperator{\Sec}{Sec}

\DeclareMathOperator{\rank}{rank}

\DeclareMathOperator{\trc}{tr}

\newcommand{\bbR}{\mathbb{R}}

\newcommand{\dd}{\mathrm{d}}

\allowdisplaybreaks[4]

\begin{document}

\title[A refined Schwarz lemma for \(V\)-harmonic maps]{A refined Schwarz lemma for \(V\)-harmonic maps}


\author{Guangwen Zhao}
\address{School of Mathematics and Statistics, Wuhan University of Technology, Wuhan 430070, China}
\curraddr{}
\email{gwzhao@whut.edu.cn}
\thanks{This work is partially supported by the National Natural Science Foundation of China (12001410)}


\subjclass[2020]{53C43, 30C80}

\keywords{Schwarz lemma, \(V\)-harmonic map, generalized dilatation}


\dedicatory{}

\begin{abstract}
    In this note, we establish a refined Schwarz lemma for \(V\)-harmonic maps. Specifically, we prove that if \(u\) is a \(V\)-harmonic map of generalized dilatation of order \(\beta \) from a complete Riemannian manifold with Bakry--\'Emery Ricci curvature bounded below by a constant \(-A\) to a Riemannian manifold with sectional curvature bounded above by a negative constant \(-B\), then
    \[
    u^*h\le \frac{Ac(\beta )}{BD(\beta )}g \le \frac{A\beta^2}{B}g,
    \]
    where \(c(\beta )=\beta^2/(1+\beta^2)\) and \(D(\beta )=1-\lfloor 1/c(\beta )\rfloor c(\beta )^2-(1-\lfloor 1/c(\beta )\rfloor c(\beta ))^2\). Equality in the second inequality holds if and only if \(\beta =1,\ 1/\sqrt{2},\ 1/\sqrt{3}, \cdots \). Our result improves the previous bounds obtained by Shen (J. Reine Angew. Math., 1984) for harmonic maps and by Chen--Li--Qiu (Nonlinear Anal., 2022) for \(f\)-harmonic maps. We also present some applications of our main theorem.
\end{abstract}

\maketitle

\tableofcontents

\section{Introduction}

The classical Schwarz--Pick lemma states that any holomorphic map from the unit disk in the complex plane into itself is distance-decreasing with respect to the Poincar\'e metric. Ahlfors \cite{ahlfors1938} extended this to holomorphic maps into Riemann surfaces with negative curvature, revealing that the phenomenon is driven by a comparison of domain and target curvatures. Chern \cite{chern1968} and Lu \cite{lu1968} then extended this result to more general dimensions and targets. In 1978, Yau \cite{yau1978} established a general Schwarz lemma for holomorphic maps from a complete K\"ahler manifold with Ricci curvature bounded below into a Hermitian manifold with holomorphic bisectional curvature bounded above by a negative constant, he proved that such maps are distance decreasing up to a constant depending only on these curvature bounds. Later, this result was subsequently refined by Royden \cite{royden1980}, who relaxed the target curvature condition to holomorphic sectional curvature, and by Yang--Chen \cite{yc1984} for complete Hermitian manifolds. In \cite{tosatti2007}, Tosatti treated the case where the almost complex structure is nonintegrable. Under curvature and torsion conditions on the canonical connection, he established a Schwarz lemma for holomorphic maps between almost Hermitian manifolds. In recent years, there is a rich body of work on this topic, which relaxes the curvature assumptions or broadens the manifold settings, see \cite{bs2026,cn2022,cdry2021,dry2021,ni2021,savas-halilaj2019,yang2021,yang2024,yu2022} and the references therein.

Another direction in extending the Schwarz lemma is to consider harmonic maps. Goldberg-Har'El \cite{gh1977} initiated this line of study in 1977. To overcome the difficulty that the curvature term of the target manifold in the Bochner formula for harmonic maps does not directly yield a lower bound proportional to the square of the energy density, they introduced the notion of bounded dilatation. They proved that harmonic maps of bounded dilatation satisfy a Schwarz-type estimate, under the assumptions that the domain has Ricci curvature bounded below and the target has sectional curvature bounded above by a negative constant. However, their bound depends on the dimensions of the domain and the target. Shen \cite{shen1984} subsequently introduced the broader notion of generalized dilatation, which is weaker than bounded dilatation. Every map of bounded dilatation automatically has generalized dilatation, whereas the converse only holds with a larger constant. By applying a local cutoff argument to the maximal eigenvalue of the pullback metric, Shen established a Schwarz lemma with a constant depending only on the curvature bounds and the dilatation order, thereby removing the dimensional dependence present in \cite{gh1977} and yielding a cleaner upper bound. In \cite{cjw2015}, Chen--Jost--Wang introduced the notion of \(V\)-harmonic maps, and in \cite{cjq2012}, Chen--Jost--Qiu developed its general framework. This concept encompasses harmonic maps under several important settings, as detailed in \cite{cjw2015}. Consequently, it is natural to investigate Schwarz lemmas for \(V\)-harmonic maps. In \cite{cz2017}, Chen and the author proved a Schwarz lemma for \(V\)-harmonic maps of generalized dilatation, where the constant depends on the dilatation order, the curvature bounds and the dimensions of the domain and the target. Later, Chen--Li--Qiu \cite{clq2022} treated the gradient case \(V=-\nabla f\) and refined the estimate in \cite{cz2017} by a local cutoff argument on the maximal eigenvalue of the pullback metric, obtaining a dimension-independent constant. However, their estimate remains less sharp compared with \cite{shen1984}. Very recently, there has also been related work on harmonic or \(V\)-harmonic maps in more general settings, see \cite{huang2024,hy2024,wang2024}.

In this paper, we continue the study of Schwarz lemmas for \(V\)-harmonic maps. Our main goal is to refine existing estimates so that they not only attain but strictly improve the sharpness of Shen’s original bound. Unlike the approaches in \cite{shen1984} and \cite{clq2022}, we do not work with the maximal eigenvalue of the pullback metric. Instead, we analyze the smooth energy density directly. By establishing a sharp algebraic inequality adapted to the generalized dilatation condition, we obtain a more refined estimate for the curvature term of the target manifold in the Bochner formula. This bypasses the regularity issue and avoids the cutoff loss associated with eigenvalue methods. Consequently, compared with \cite{clq2022}, we can treat a much broader class of vector fields \(V\) and derive significantly sharper estimates (see \eqref{eq-e} in Theorem~\ref{thm-a} and the third part in Remark~\ref{rem-a} for details). In the special case \(V=0\), although our simplified bound formally coincides with that of \cite{shen1984}, our refined estimate genuinely improves Shen’s classical constant for almost all \(\beta \) (see estimate \eqref{eq-d} in Theorem~\ref{thm-a} and the first part in Remark~\ref{rem-a} for details).

The remainder of this paper is organized as follows. In Section~\ref{sec-2}, we state and prove our main theorem, together with a detailed comparison with existing related results. Section~\ref{sec-3} is devoted to presenting applications of the main theorem.

\section{Main results}\label{sec-2}

Let \((M,g)\) and \((N,h)\) be two Riemannian manifolds, and let \(V\) be a smooth vector field on \(M\). A smooth map \(u:M\to N\) is called \(V\)-harmonic if it satisfies
\[
	\tau_V(u)=\tau(u)+\dd u(V)=0,
\] 
where \(\tau(u)\) is the tension field of \(u\). Denote \(\Delta_V=\Delta +\iota_V\circ \dd\) and the Bakry--\'Emery Ricci tensor by \(\Ric_V=\Ric -\frac{1}{2}\mathcal{L}_Vg\), where \(\iota\) is the interior product, \(\Ric \) is the Ricci tensor, and \(\mathcal{L}\) is the Lie derivative.

Let \(u:(M^m,g)\to (N^n,h)\) be a smooth map between Riemannian manifolds. Let \(\{e_i\}_{i=1}^m\) and \(\{\tilde e_\alpha \}_{\alpha =1}^n\) be local orthonormal frame fields, respectively. Denote 
\[
	\dd u(e_i)=\sum_\alpha u_i^\alpha \tilde e_\alpha ,
\]
and define 
\[
	U_{ij}=\langle \dd u(e_i),\dd u(e_j)\rangle =\sum_\alpha u_i^\alpha u_j^\alpha . 
\]
In other words, we have \(U=(u_i^\alpha )(u_i^\alpha )^\top \), where \((u_i^\alpha )^\top \) denotes the transpose of \((u_i^\alpha )\). Hence \(U=(U_{ij})\) is a \(m\times m\) symmetric positive semi-definite matrix. Therefore, the eigenvalues of \(U\) are nonnegative, ordered as 
\[
	\lambda_1\ge \lambda_2\ge \cdots \ge \lambda_m\ge 0.
\]
Recall that a map \(u:M\to N\) is called generalized dilatation of order \(\beta \), if there exists a positive constant \(\beta >0\) such that 
\begin{equation}\label{eq-z}
	\lambda_1(x)\le \beta^2(\lambda_2(x)+\cdots +\lambda_m(x))
\end{equation}
holds for every \(x\in M\) (see \cite{gh1977,shen1984}). Hence, imposing the generalized dilatation condition requires \(\dim M\ge 2\) and \(\dim N\ge 2\). 

Let \(k=\min \{m,n\}\), so that \(\rank U\le k\). With these notational conventions in place, we denote
\[
	e(u):=|\dd u|^2=\trc U=\sum_{i=1}^k\lambda_i,\qquad 
	Q(u):=(\trc U)^2-\trc (U^2)=e(u)^2-\sum_{i=1}^k\lambda_i^2=2\sum_{i<j}\lambda_i\lambda_j=2\left |{\bigwedge}^2\dd u\right |^2,
\]
and we have
\[
	u^*h\le \lambda_1g.
\]
Moreover, the definition of the generalized dilatation of order \(\beta \), namely \eqref{eq-z}, can be reformulated as 
\begin{equation}\label{eq-y}
	\lambda_1\le \beta^2(\lambda_2+\cdots +\lambda_k).
\end{equation}

We now state the main theorem of this paper. For \(\beta >0\), define 
\begin{equation}\label{eq-c}
	c(\beta )=\frac{\beta^2}{1+\beta^2},\qquad p(\beta )=\left \lfloor \frac{1}{c(\beta )}\right \rfloor,\qquad D(\beta )=1-p(\beta )c(\beta )^2-(1-p(\beta )c(\beta ))^2.
\end{equation}

\begin{theorem}\label{thm-a}
	Let \((M,g)\) be a complete Riemannian manifold with Bakry--\'Emery Ricci curvature \(\Ric_V\ge -A\), where \(A\ge 0\) is a constant and \(V\) is a smooth vector field on \(M\). Let \((N,h)\) be a Riemannian manifold with sectional curvature \(\Sec^N\le -B\), where \(B>0\) is a constant. Let \(u:M\to N\) be a \(V\)-harmonic map of generalized dilatation of order \(\beta \). Then
	\begin{equation}\label{eq-d}
	    u^*h\le \frac{Ac(\beta )}{BD(\beta )}g.
	\end{equation}
	In particular, 
	\begin{equation}\label{eq-e}
		u^*h\le \frac{A\beta^2}{B}g.
	\end{equation}
	If \(\beta \ge 1\), then 
	\begin{equation}\label{eq-f}
		u^*h\le \frac{A(1+\beta^2)}{2B}g.
	\end{equation}
\end{theorem}

\begin{remark}\label{rem-a}
    \begin{enumerate}[leftmargin=*,label=\normalfont(\arabic*)]
    \item When \(V=0\), our simplified bound \eqref{eq-e} coincides with that of Shen \cite[Main theorem]{shen1984}. However, our precise estimate \eqref{eq-d} actually strictly improves Shen's result in almost all cases. As will be seen in the proof below, the equality
    \[
    \frac{c(\beta )}{D(\beta )}=\beta^2 
    \]
    holds if and only if \(D(\beta )=1-c(\beta )\). By Lemma~\ref{lem-a}, this requires \(1/\beta^2\) to be an integer, that is, \(\beta =1,\ 1/\sqrt{2},\ 1/\sqrt{3}, \cdots \). For all other noninteger values of \(1/\beta^2\), we have \(D(\beta )>1-c(\beta )\), which makes our constant in \eqref{eq-d} strictly smaller than \(A\beta^2/B\). Consequently, Theorem~\ref{thm-a} extends and strictly improves Shen's main theorem.

    \item The estimate of Chen--Zhao \cite[Theorem~3.3]{cz2017} is 
    \[
	u^*h\le \frac{Ak^2\beta^4}{2B(1+\beta^2)}g.
    \]
    At a noncritical point, \eqref{eq-y} forces \(\beta^2\ge 1/(k-1)\). Hence the ratio of this constant to the constant in \eqref{eq-e} is at least \(k/2\). In other words, compared with the estimate of Chen–Zhao, our constants are not only independent of \(\dim M\) and \(\dim N\), but are indeed strictly smaller. 

    \item In the case \(V=-\nabla f\), Chen--Li--Qiu \cite[Theorem~1]{clq2022} obtained 
    \[
	u^*h\le \frac{3A\beta^2}{B}g
    \]
    by applying a local cutoff argument directly to \(\lambda_1\). Exploiting the symmetry of \(\frac{1}{2}\mathcal{L}_{\nabla f}g=\Hess f\), they derived from the Bochner formula the differential inequality 
    \[
	\Delta_V\lambda_1\ge 2\left(B\left(\lambda_1\sum\lambda_i-\lambda_1^2\right)-A\lambda_1\right)\ge 2\left(\frac{B}{\beta^2}\lambda_1^2-A\lambda_1\right),
    \]
    which would yield the bound \(A\beta^2/B\) if the cutoff errors vanished. However, under the Bakry--\'Emery Ricci curvature condition, the weighted Laplacian comparison forces \(\Delta_Vr\) to grow linearly, so the distance cutoff error does not vanish as the cutoff radius tends to infinity, enlarging the constant to \(3A\beta^2/B\). The smooth trace \(e(u)=\sum_{i=1}^k\lambda_i\) avoids the eigenvalue regularity issue and the cutoff loss, while also applying to a general vector field \(V\). Our Theorem~\ref{thm-a} therefore not only covers but also strictly improves their result.
    \end{enumerate}
\end{remark}

\begin{remark}
	When \(A=0\), Theorem~\ref{thm-a} directly yields a Liouville-type result: there exists no nonconstant \(V\)-harmonic map of generalized dilatation from a complete Riemannian manifold with nonnegative Bakry--\'Emery Ricci curvature to a Riemannian manifold whose sectional curvature admits a negative upper bound.
\end{remark}

Before proving Theorem~\ref{thm-a}, we establish two preparatory lemmas. The first one is a sharp algebraic inequality used to bound the curvature term of the target manifold in the Bochner formula from below by the square of the energy density, and the second one is an analytic lemma used to extract a uniform upper bound for the energy density from its differential inequality.

\begin{lemma}\label{lem-a}
	Let \(x_1,\cdots ,x_k\ge 0\), \(\sum_{i=1}^kx_i=1\), and \(\max x_i\le c<1\) for some \(c\). Put 
	\[
	p=\left \lfloor \frac{1}{c}\right \rfloor,\qquad D(c)=1-pc^2-(1-pc)^2.
	\]
	Then 
	\[
	1-\sum_{i=1}^kx_i^2\ge D(c).
	\]
	Equality holds if and only if
	\[
	(x_1,\cdots ,x_k)=(\underbrace{c,\cdots ,c}_{p\ \mbox{times}},1-pc,0,\cdots ,0)
	\]
	up to a permutation, provided that the constrained simplex is nonempty.
\end{lemma}

\begin{proof}
	If \(c<1/k\), the feasible set is empty, and the conclusion holds vacuously. We assume \(c\ge 1/k\). Set 
	\[
	K=\left\{x\in \bbR^k: 0\le x_i\le c,\ \sum_{i=1}^kx_i=1\right\}.
	\]
	The set \(K\) is a nonempty closed and bounded subset of \(\bbR^k\), hence it is compact. Since the function \(f(x)=\sum_{i=1}^kx_i^2\) is continuous, \(f\) attains a global maximum on \(K\) at some point \(x^*=(x_1^*\cdots x_k^*)\). Suppose there exist two indices \(i\ne j\) such that \(0<x_i^*\le x_j^*<c\). Set \(\delta =\min\{x_i^*,c-x_j^*\}>0\). For any \(t\in [0,\delta ]\), define a perturbed vector \(x(t)\in \bbR^k\) by
	\[
	x_i(t)=x_i^*-t,\qquad x_j(t)=x_j^*+t,\qquad x_\ell (t)=x_\ell ^*\ \mbox{for all } \ell \ne i,j. 
	\]
	Since \(0\le x_i^*-t<c\) and \(0<x_j^*\le x_j^*+t\le x_j^*+(c-x_j^*)=c\), the perturbed vector field satisfies \(0\le x_\ell (t)\le c\) for all \(\ell \). Moreover, 
	\[
	\sum_{\ell =1}^kx_\ell(t)=(x_i^*-t)+(x_j^*+t)+\sum_{\ell \ne i,j}x_\ell ^*=\sum_{\ell =1}^kx_\ell ^*=1.
	\]
	Thus \(x(t)\in K\) for all \(t\in [0,\delta ]\). The change in the objective value,
	\begin{align*}
        \sum_{\ell =1}^kx_\ell (t)^2-\sum_{\ell =1}^k(x_\ell ^*)^2
        =&(x_i^*-t)^2+(x_j^*+t)^2-(x_i^*)^2-(x_j^*)^2\\
        =&2t(x_j^*-x_i^*)+2t^2,
	\end{align*}
	is strictly positive whenever \(t\in (0,\delta ]\). This contradicts the maximality of \(x^*\). Therefore, at any global maximum point, it is impossible to have two distinct coordinates strictly between \(0\) and \(c\). Consequently, at most one coordinate can lie in the open interval \((0,c)\), while all other coordinates must belong to the boundary \(\{0,c\}\).

	Based on the above discussion, at a global maximum point \(x^*\), there are \(q\) coordinates equal to \(c\), one coordinate equal to \(x_r\in [0,c)\) for some \(r\in \{1,\cdots ,k\}\) (we require \(x_r<c\) because if \(x_r=c\) it would be counted among the \(q\) coordinates equal to \(c\)), and the remaining \(k-q-1\) coordinates equal to \(0\). The case of no coordinate in \((0,c)\) corresponding to \(x_r\in \{0,c\}\). The constraint \(\sum_{i=1}^kx_i=1\) gives \(x_r=1-qc\). Feasibility requires \(0\le x_r<c\), which is equivalent to 
	\[
	\frac{1}{c}-1<q\le \frac{1}{c}.
	\]
	Since \(q\) is a nonnegative integer and the interval \((1/c-1,1/c]\) has length \(1\), it contains exactly one integer, namely \(q=\lfloor 1/c\rfloor =p\). Consequently, the objective value is
	\[
	\sum_{i=1}^k(x_i^*)^2=pc^2+(1-pc)^2,
	\]
	achieved by \(p\) coordinates equal to \(c\), one coordinate equal to \(1-pc\), and the remaining \(k-p-1\) coordinates equal to \(0\). We complete the proof.
\end{proof}

\begin{lemma}\label{lem-b}
	Let \(w\ge 0\) be a smooth function on a Riemannian manifold \((M,g)\) and let \(V\) be a smooth vector field on \(M\). Suppose that, for some \(a\ge 0,\ b>0,\ \alpha >0\) and \(E_0\ge 0\), 
	\begin{equation}\label{eq-a}
		\frac{1}{2}\Delta_Vw\ge -aw+bw^{1+\alpha }
	\end{equation}
	on \(\{w>E_0\}\). Then
	\[
	\sup_Mw\le \max\left\{E_0,\left(\frac{a}{b}\right)^{1/\alpha }\right\}.
	\]
\end{lemma}

\begin{proof}
	Fix \(C>0\), and define \(F=-(w+C)^{-\alpha /2}\). Since \(w\ge 0\), we have \(-C^{\alpha /2}\le F<0\), that is, \(F\) is bounded. By \cite[Theorem~1]{cq2016}, there exist points \(\{x_j\}\subset M\) such that
	\[
	\lim_{j\to \infty }F(x_j)=\sup_MF,\qquad \lim_{j\to \infty }|\nabla F|(x_j)=0,\qquad \lim_{j\to \infty }\Delta_VF(x_j)\le 0.
	\]
	From the definition of \(F\), we know that \(\{x_j\}\) is also the maximizing sequence for \(w\). Now, differentiating \(F\), we obtain
	\[
	\nabla F=\frac{\alpha }{2}(w+C)^{-\alpha /2-1}\nabla w.
	\]
	By calculating the divergence and then combining it with the drift term \(\langle V,\nabla F\rangle \), we arrive at
	\[
    \Delta_VF=-\frac{\alpha }{2}\left(\frac{\alpha }{2}+1\right)(w+C)^{-\alpha /2-2}|\nabla w|^2+\frac{\alpha }{2}(w+C)^{-\alpha /2-1}\Delta_Vw,
	\]
	which is equivalent to 
	\begin{equation}\label{eq-b}
	    \frac{\Delta_Vw}{(w+C)^{\alpha +1}}=\frac{2}{\alpha }\frac{\Delta_VF}{(w+C)^{\alpha /2}}+\frac{4}{\alpha^2}\left(\frac{\alpha }{2}+1\right)|\nabla F|^2.
	\end{equation}
	If \(w(x_j)=\sup_Mw=+\infty \), the right-hand side of \eqref{eq-b} has nonpositive upper limit, while dividing \eqref{eq-a} by \((w+C)^{\alpha +1}\) yields a positive lower limit \(2b\). This is impossible, and therefore \(w(x_j)=\sup_Mw<+\infty \). If \(\sup_Mw>E_0\), the sequence \(\{x_j\}\) eventually lies in \(\{w>E_0\}\), since \(w+C\) is then bounded along \(\{x_j\}\), \eqref{eq-b} yields
	\[
	-a\sup_Mw+b\left(\sup_Mw\right)^2\le 0.
	\]
	Hence \(\sup_Mw\le (a/b)^{1/\alpha }\). If \(\sup_Mw\le E_0\), there is nothing to prove.
\end{proof}

With the preparatory lemmas proved, we are now in a position to prove Theorem~\ref{thm-a}.

\begin{proof}[Proof of Theorem~\ref{thm-a}]
    Let \(\{e_i\}\) a local orthonormal frame field of \(M\). Applying the Bochner formula for \(V\)-harmonic maps, together with the definition of the matrix \(U\), we have
    \begin{equation}\label{eq-g}
        \begin{split}
        	\frac{1}{2}\Delta_Ve(u)
        	=&|\nabla \dd u|^2+\sum_{i,j}\Ric_V(e_i,e_j)\langle \dd u(e_i),\dd u(e_j)\rangle -\sum_{i,j}\langle R^N(\dd u(e_i),\dd u(e_j))\dd u(e_j),\dd u(e_i)\rangle \\
        	=&|\nabla \dd u|^2+\sum_{i,j}\Ric_V(e_i,e_j)U_{ij}-\sum_{i,j}\Sec^N(\dd u(e_i),\dd u(e_j))\left(U_{ii}U_{jj}-U_{ij}^2\right).
        \end{split}
    \end{equation}
    Since \(U\) is a symmetric positive semi-definite matrix and \(\Ric_V\ge -Ag\), we obtain
    \[
    \sum_{i,j}\Ric_V(e_i,e_j)U_{ij}\ge -A\trc U=-Ae(u).
    \]
    On the other hand, the symmetric positive semi-definiteness of \((U_{ij})\) also implies \(U_{ii}U_{jj}-U_{ij}^2\ge 0\). Combining this with \(\Sec^N\le -B\), we deduce
    \begin{align*}
    	-\sum_{i,j}\Sec^N(\dd u(e_i),\dd u(e_j))\left(U_{ii}U_{jj}-U_{ij}^2\right)
    	\ge &B\sum_{i,j}\left(U_{ii}U_{jj}-U_{ij}^2\right)\\
    	=&B\left((\trc U)^2-\trc (U^2)\right)\\
    	=&BQ(u).
    \end{align*}
    Therefore, by substituting the above two inequalities into \eqref{eq-g}, we obtain
    \begin{equation}\label{eq-h}
    	\frac{1}{2}\Delta_Ve(u)\ge -Ae(u)+BQ(u).
    \end{equation}
    Set \(x_i=\lambda_i/\sum_{i=1}^k\lambda_i\) for \(i=1,\cdots ,k\). Then \(\sum_{i=1}^kx_i=1\). At every point where \(e(u)=\sum_{i=1}^k\lambda_i>0\), \eqref{eq-y} gives \(\lambda_1\le c(\beta )\sum_{i=1}^k\lambda_i\), which implies \(x_i\le c(\beta )<1\) for all indices \(i\). Applying Lemma~\ref{lem-a}, then gives
    \[
    Q(u)=\left(1-\sum_{i=1}^kx_i^2\right)\left(\sum_{i=1}^k\lambda_i\right)^2\ge D(\beta )\left(\sum_{i=1}^k\lambda_i\right)^2=D(\beta )e(u)^2.
    \]
    Thus we deduce
    \begin{equation}\label{eq-i}
    	-\sum_{i,j}\Sec^N(\dd u(e_i),\dd u(e_j))\left(U_{ii}U_{jj}-U_{ij}^2\right)
    	\ge BD(\beta )e(u)^2.
    \end{equation}
    Substituting \eqref{eq-i} into \eqref{eq-h} yields
    \[
    \frac{1}{2}\Delta_Ve(u)\ge -Ae(u)+BD(\beta )e(u)^2.
    \]
    By using Lemma~\ref{lem-b} with \(E_0=0\) and \(\alpha =1\), we obtain
    \[
    e(u)\le \frac{A}{BD(\beta )},
    \]
    and hence \eqref{eq-d} holds:
    \[
    u^*h\le \lambda_1g\le c(\beta )e(u)g\le \frac{Ac(\beta )}{BD(\beta )}g.
    \]
    Moreover, 
    \[
    \sum_{i=1}^kx_i^2\le c(\beta )\sum_{i=1}^kx_i=c(\beta ),
    \]
    so \(D(\beta )=\min \left(1-\sum_{i=1}^kx_i^2\right)=1-\max \sum_{i=1}^kx_i^2\ge 1-c(\beta )\). Consequently,
    \[
    \frac{c(\beta )}{D(\beta )}\le \frac{c(\beta )}{1-c(\beta )}=\beta^2,
    \]
    and thus \eqref{eq-e} holds. When \(\beta \ge 1\), we have \(p(\beta )=\lfloor 1+1/\beta^2\rfloor =1\), whence \(D(\beta )=2\beta^2/(1+\beta^2)^2\) immediately. Consequently, \(c(\beta )/D(\beta )=(1+\beta^2)/2\), and now \eqref{eq-f} follows directly from \eqref{eq-d}.
\end{proof}

\section{Applications}\label{sec-3}

Once Theorem~\ref{thm-a} is applied to \(V\)-harmonic maps of generalized dilatation on Ricci solitons, one yields Schwarz-type estimates or Liouville-type results. Recall that a Riemannian manifold \((M,g)\) is called a Ricci soliton if there exist a smooth vector field \(V\) on \(M\) and a constant \(\rho \) such that \(\Ric_V=\Ric -\frac{1}{2}\mathcal{L}Vg=\rho g\). A Ricci soliton is side to be shrinking, steady, and expanding, if \(\rho >0,\ =0\), and \(<0\), respectively. Then Theorem~\ref{thm-a} implies the following corollary.

\begin{corollary}
	Let \((M,g,V,\rho )\) be a complete Ricci soliton and let \((N,h)\) be a Riemannian manifold with sectional curvature bounded above by a negative constant \(-B\). Let \(u:M\to N\) be a \(V\)-harmonic map of generalized dilatation of order \(\beta \). Then,
	\begin{enumerate}
		\item if \((M,g,V,\rho )\) is shrinking or steady, then \(u\) is a constant map.
		\item if \((M,g,V,\rho )\) is expanding, then
		\[
		u^*h\le -\frac{\rho \beta^2}{B}g.
		\]
	\end{enumerate}
\end{corollary}

The classical uniformization theorem implies that any compact Riemann surface of genus at least two admits the Poincar\'e metric. Consequently, we have the following Liouville-type result, which extends \cite[Corollary~3.7]{zhao2020} from \(V\)-harmonic morphisms (horizontally weakly conformal \(V\)-harmonic maps) to \(V\)-harmonic maps of generalized dilatation.

\begin{corollary}
	Let \((M,g,V,\rho )\) be a complete shrinking or steady Ricci soliton. Then there exists no nonconstant \(V\)-harmonic map of generalized dilatation from \((M,g,V,\rho )\) to a compact Riemann surface of genus at least two. 
\end{corollary}

We take holomorphic maps between almost Hermitian manifolds as another application, since they naturally satisfy the generalized dilatation condition. Let \((M^{2m},J,g)\) and \((N^{2n},J',h)\) be two almost Hermitian manifolds. A smooth map \(u:M\to N\) is said to be holomorphic if \(\dd u\circ J=J'\circ \dd u\). A holomorphic map between almost Hermitian manifolds automatically satisfy the generalized dilatation with order \(\beta =1\) (see \cite{cz2017,gh1979}). Moreover, when the target manifold \(N\) is quasi-K\"ahler, any holomorphic map is \(V\)-harmonic with \(V=-J\delta J\) (see \cite{cz2017}). Consequently, we obtain the following result.

\begin{corollary}
	Let \((M^{2m},J,g)\) be a complete almost Hermitian manifold with \(\Ric +\frac{1}{2}\mathcal{L}_{J\delta J}g\ge -A\), where \(A\ge 0\) is a constant. Let \((N^{2n},J',h)\) be a quasi-K\"ahler manifold with sectional curvature bounded above by a negative constant \(-B\). Let \(u:M\to N\) be a holomorphic map. Then 
	\[
	u^*h\le \frac{A}{B}g.
	\]
	In particular, if \(A=0\), then there is no nonconstant holomorphic map from \(M\) to \(N\).
\end{corollary}

\begin{remark}
	Note that in the above corollary, we always have \(k=\min\{2m,2n\}\ge 2\) (when the real dimension of the domain manifold is \(2\), \(V=0\) and \(u\) is harmonic). When neither \(M\) nor \(N\) is of real dimension \(2\), this corollary strictly improves \cite[Theorem~4.3]{cz2017}, where the conclusion obtained is \(u^*h\le \frac{k^2A}{4B}g\).
\end{remark}

When the target manifold is an almost K\"ahler manifold whose bisectional curvature is bounded above by a negative constant, one observes that \cite[Theorem~4.4]{cz2017} already gives
\[
u^*h\le \frac{A}{B}g.
\]
This is because, in this setting, the relation between bisectional curvature and sectional curvature allows one to obtain directly a dimension-independent estimate for \(Q(u)\), making Lemma~\ref{lem-a} unnecessary. We therefore omit the details here. Furthermore, it is worth mentioning that when the target manifold is K\"ahler, the curvature condition can be further weakened to requiring only that the holomorphic sectional curvature admit a strictly negative upper bound. In this situation, an estimate of the form \(Q(u)\gtrsim e(u)^2\) can be achieved by appealing to Royden's lemma (\cite[p. 552]{royden1980}), without requiring the map to be of generalized dilatation.


\end{document}